\documentclass[a4paper, leqno,reqno]{amsart}

\usepackage{a4wide}
\usepackage{hyperref}
\usepackage{amssymb}
\usepackage{mathrsfs}

\numberwithin{equation}{section}
\newtheorem{theorem}{Theorem}[section]
\newtheorem{lemma}[theorem]{Lemma}

\newtheorem{remark}[theorem]{Remark}
\allowdisplaybreaks

\begin{document}
\title[On the control of solutions of viscoelastic plate problem]{On the Control of Solutions of a Viscoelastic Plate Problem with a Frictional Damping Term}
\author[ Bilel Madjour and Amel Boudiaf ]{Bilel Madjour and Amel Boudiaf}
\address{Bilel Madjour\\
 Laboratory of Applied Mathmatics (LaMA), Department of Mathematics, Faculty of Science, University of setif, Algeria}
\email{bilel.madjour@univ-setif.dz}
\address{Amel Boudaif\\
Laboratory of Applied Mathmatics (LaMA), Department of Mathematics, Faculty of Science, University of setif, Algeria}
\email{amel.boudiaf@univ-setif.dz}
\maketitle

    \begin{abstract}
      In this article, we study the stability of solutions to a nonlinear viscoelastic plate problem with frictional damping of a memory on a part of the boundary, and a logarithmic source in a bounded domain $\Omega \subset \mathbb{R}^2.$ In this problem the relaxation function $r$ satisfies $r^{\prime}\left( t\right)\leq -h\left( t\right) \mathcal{G}\left( r\left( t\right) \right)$ for all $t\geq 0$, where $h$ is a nonincreasing positive function. This work extends previous works with viscoelastic plate problems and improves earlier results in the literature.
   \end{abstract}
	
	\section{\protect\Large Introduction}
	In this study, we focus on the decay of the energy solution for the following model:
	\begin{equation}
		(P)  \left\lbrace 
		\begin{array}{c c}
			 \label{1} u_{tt}-\mu\Delta u_{tt}+\Delta^{2} u-\int_{0}^{t} r\left( t-\tau\right) \Delta^{2} u\left( \tau\right) d\tau=\left\vert u\right\vert^{p-2}u\ln\left\vert u\right\vert, &\text{in\ }\Omega \times \left( 0,\infty \right), \\
			\Phi_{1}\left( u\right) -\Phi_{1}\left(\int_{0}^{t} r\left(t-\tau\right)u\left( \tau\right) d\tau\right)=0, & \text{on } \Gamma\times \left(0,\infty \right), \\
			\Phi_{2}\left( u\right)-\mu\dfrac{\partial u_{tt}}{\partial \nu} -\Phi_{2}\left(\int_{0}^{t} r\left(t-\tau\right)u\left( \tau\right) d\tau\right)=0, & \text{on } \Gamma\times \left(0,\infty \right), \\		
			u\left( x,0\right) =u_{0}\left( x\right),\text{ }u_{t}\left( x,0\right) =u_{1}\left( x\right), & x\in \Omega.\notag
		\end{array}\right.
	\end{equation} 
	The domain $\Omega$ is an open-bounded set of $\mathbb{R}^{2}$ with a smooth boundary $\Gamma,$  and $p>2, \mu>0.$ The term responsible for the viscoelastic damping appears as an integral term in the differential equation $\left(P\right)_{1},$ representing the memory. The differential operators $\Phi_{1}$ and $\Phi_{2}$ are defined as follows:
	\begin{equation*}
		\Phi_{1}u=\Delta u+\left(1-\rho\right) D_{1}u,\quad \text{ and } \quad \Phi_{2} u=\frac{\Delta u}{ \partial\nu }+\left(1-\rho\right) \frac{\partial\left( D_{2}u\right)}{\partial \eta} ,
	\end{equation*}
	where $\rho$ denotes Poisson's ratio, $\left( \text{with } 0<\rho<\frac{1}{2} \right) $  and
	\begin{equation}
	D_{1}u=2\nu_{1}\nu_{2}u_{x_{1}x_{2}}-\nu_{1}^{2}u_{x_{2}x_{2}}-\nu_{2}^{2}u_{x_{1}x_{1}},\quad D_{2}u=\left( \nu_{1}^{2}-\nu_{2}^{2}\right) u_{x_{1}x_{2}}-\nu_{1}\nu_{2}\left( u_{x_{1}x_{1}}-u_{x_{2}x_{2}}\right).\notag
	\end{equation}
	The vector $\nu=\left( \nu_{1},\nu_{2}\right) $ denotes the external unit normal on $\Gamma,$ and $\eta=\left( -\nu_{2},\nu_{1}\right)$ represents the unit tangent vector obtained from $\nu$ through a $\pi/2$ rotation.
	\hspace{0.5cm}
	\par \textbf{Viscelastic Problems.} The significance of the viscoelastic properties of materials has been increasingly recognized due to the rapid advancements in the rubber and plastics industries. The primary interest in viscoelastic materials lies in their natural damping capabilities, attributed to their ability to retain a memory of their deformation history.  Mathematically, these damping effects are described using integro-differential operators.
	\hspace{0.5cm}
	\par \textbf{Plate problems} generally refer to mathematical and engineering problems related to the behavior and analysis of plates, which are thin and flat structural elements. These issues are frequently studied in solid mechanics and structural engineering, including material engineering, mechanical engineering,  nuclear physics, and quantum mechanics \cite{f,g}, geophysics and optics \cite{c,d,e}. The plate problem involves understanding the deformation, stress and stability of plates subjected to diverse conditions and loads.
	\hspace{0.5cm}
	\par\textbf{Plate problems with a Logarithmic source term.} In 2023, Islam et al. \cite{6} considered the following Euler-Bernoulli equation
	\begin{equation*}	
		u_{tt}+\Delta^{2} u-\int_{0}^{t}r\left( t-\tau\right) \Delta^{2} u\left(\tau\right) d\tau +h\left( u_{t}\right) =\left\vert u\right\vert ^{\gamma-2}u\ln{\left\vert u\right\vert}, 
	\end{equation*}
	where $\gamma>2$, the relaxation function $r$ realized $(\ref{789})$, they provide a general decay result in the presence of nonlinear feedback on part of the boundary. In 2022, Han and Li \cite{a} devoted their attention to the initial boundary value system associated with a damped plate equation described by
	\begin{equation*}
		\begin{cases}
			u_{tt}+\Delta^{2} u-\Delta u-\Delta u_{t}+u_{t}=\left\vert u\right\vert ^{p-2}u\ln{\left\vert u\right\vert},& \left(x,t\right)\in\Omega \times \left( 0,T \right), \\
			u\left(x,t\right)=\Delta u(x,t)=0, & (x,t)\in\partial\Omega\times \left( 0,T \right),\\
			u\left( x,0\right) =u_{0}\left(x\right),\text{ }u_{t}\left( x,0\right) =u_{1}\left(x\right), & x\in \Omega,
		\end{cases}
	\end{equation*}
	where $T$ belongs to $ \left( 0, +\infty\right]$ and represents the maximal existence time of the solution $u(x,t)$, and $\Omega$ is a bounded domain in $\mathbb{R}^{n} (n\geq 1)$  with a smooth boundary $\partial \Omega$. The exponent $p$ satisfies the condition   
	\begin{equation*}
		\left( A\right)\hspace{0.3cm} 2 < p < 2_{*},
	\end{equation*}
	where $2_{*} = +\infty$ if $n \leq 4$ and $2_{*}=\frac{2n}{n-4}$ for $n \geq 5$.\\
	The authors determined finite time blow-up criteria for systems with low and high initial energy levels. In 2022, Al-Gharabli et al.\cite{007} examined the following viscoelastic plate system and derived a general decay result
	\begin{equation*}
		u_{tt}+\Delta^{2} u+ u-\int_{0}^{t}r\left( t-\tau\right) \Delta^{2}u\left( \tau\right) d\tau+h\left( u_{t}\right) =ku\ln{\left\vert u\right\vert},\quad \left( x,t\right) \in\Omega \times \left( 0,\infty \right),
	\end{equation*}
	where the relaxation function $r$ satisfies $\left( H_{1}\right) $ and $\left( H_{2}\right).$ For more information see \cite{x0,x1,x6}.    
	\par Motivated by their results, we considered a new decay results of the solutions for the problem $\left( P\right) $. In this contribution we allowed viscoelasticity and logharithmic nonlinearity source. This work represents an extension of the results obtained in \cite{007,x1,x6}. 
	\section{\protect\Large Preliminaries} 
	In this section, we introduce fundamental material necessary for presenting our results. We denote
	\begin{equation*}
	\left< u,v\right> _{\Omega}=\int_{\Omega}uvdx,\quad \left< u,v\right> _{\Gamma}=\int_{\Gamma}uvd\Gamma.
	\end{equation*} 
	For simplicity, we denote $\left\Vert. \right\Vert_{L^{2}\left( \Omega\right) }$ and $\left\Vert . \right\Vert_{L^{2}\left( \Gamma\right) }$ by $\left\Vert . \right\Vert_{ \Omega }$ and $\left\Vert . \right\Vert_{\Gamma}$, respectively. The solution to problem $\left( P\right)$
	must belong to the Sobolev space $H^{2}\left( \Omega \right).$\\
	We introduce the embedding $H^{2}\left( \Omega \right)\hookrightarrow L^{q}\left( \Omega\right) $ for $q\geq 2$. Let $B_{q}>0$ the Sobolev embedding optimum constant that satisfies the inequality
	\begin{equation}
		\left\Vert u\right\Vert_{L^{q}(\Omega)}\leq B_{q}\left\Vert \nabla u\right\Vert_{L^{2}(\Omega)}. \;\label{103} 
	\end{equation}
	Using Green’s formula and simple calculations, for any $u\in H^{4}\left( \Omega \right)$ and $w\in H^{2}\left( \Omega \right),$ we get
	\begin{equation}
		\left< \Delta^{2}u,v\right>_{\Omega }= a\left(u,v\right) +\left< \Phi_{2}u,v\right>_{\Gamma } -\left< \Phi_{1}u,\frac{\Delta v}{\partial\nu}\right>_{\Gamma },\label{107}
	\end{equation}
	where the form $a\left(u,u\right) $ is bilinear and symmetric, defined as follows:
	\begin{equation*}
		a\left(u,v\right)=\int_{\Omega}  \left\lbrace u_{x_{1}x_{1}}v_{x_{2}x_{2}}+u_{x_{2}x_{2}}v_{x_{1}x_{1}}+\left(2-\rho\right)\left(u_{x_{1}x_{1}}v_{x_{1}x_{1}}+ u_{x_{2}x_{2}}v_{x_{2}x_{2}}\right)+2\left(1-\rho\right)u_{x_{1}x_{2}}v_{x_{1}x_{2}}\right\rbrace dx.
	\end{equation*}
	By Korn's inequality, we have that the functional $\sqrt{a\left(.,.\right)}$ defines a norm in $H^{2}\left( \Omega \right)$ equivalent to the usual norm of $H^{2}\left( \Omega\right) .$ This means,
	for that some positive constants $c_{0}$ and $c_{1}$, we have
	\begin{equation}
		c_{0}\left\Vert v\right\Vert_{H^{2}\left( \Omega\right) }\leq a\left( v,v\right)\leq  c_{1} \left\Vert v\right\Vert_{H^{2}\left( \Omega\right)},\label{104}
	\end{equation}
	so 
	\begin{equation}
		c_{0}\left\Vert \nabla v\right\Vert_{\Omega}^{2}\leq a\left( v,v\right).\label{105}
	\end{equation}
	For any $\sigma > 0,$ Young's inequality gives, 
	\begin{equation}
		a\left(u,v\right) \leq \frac{\sigma}{2} a\left(u,u\right) +\frac{1}{2\sigma}a\left(v,v\right).\label{106}
	\end{equation}
	\par We take into account the following assumption:\\
	$\left(H_{1}\right)$ The kernel $r$ is a $C^{1}\left( \mathbb{R}^{+}\right) $ positive function satisfying
	\begin{equation} 
		1-\int_{0}^{+\infty}r\left(\tau\right) d\tau=l>0,\quad r\left( 0\right) >0.\label{842}
	\end{equation}
	$\left(H_{2}\right)$ There exists a  nonincreasing differentiable positive function $h:\mathbb{R}^{+}\rightarrow\mathbb{R}^{+}$ with $h\left( 0\right) \geq r_{0}$, and $C^{1}$ function $\mathcal{G}:\mathbb{R}^{+}\rightarrow\mathbb{R}^{+}$ satisfies 
	\begin{equation}	
		\mathcal{G}\left( 0 \right)=\mathcal{G}^{\prime}\left( 0\right)=0, \quad\text{ and }\quad r^{\prime }\left( t\right)\leq-h\left( t\right)\mathcal{G}\left(  r\left(t\right)\right), \label{789}
	\end{equation}
	and $\mathcal{G}$ is a linear or strictly increasing and strictly convex $C^{2}$ function on $\left( 0,r_{0}\right] ,0<r_{0}<1.$
	\par For completeness we state, without proof, the following standard existence result.
	\begin{theorem}
		Let $u_{0}\in H^{2}\left(\Omega\right)$ and $u_{1}\in H^{1}\left(\Omega\right)$ be given initial data. Assume that the conditions $\left(H_{1}\right)$ and $\left(H_{2}\right)$ hold. Then, there exists a weak solution $u$ to the problem $\left(P\right)$ such that
		\begin{equation*}
			u\in L^{\infty}\left(0,T; H^{2}\left(\Omega\right)\right)\cap L^{2}\left(0,T; H^{4}\left(\Omega\right)\right),
		\end{equation*}
		for some $T>0.$
	\end{theorem}
	\section{\protect\Large Decay of solution}	
	In this section, we introduce certain functionals and lemmas instrumental in determining the uniform decay of $E\left( t\right)$ for the model $\left(P\right)$. Initially, 
	\begin{equation*}
		\left(r\circ u\right)\left( t\right)=\int_{0}^{t}r\left( t-\tau\right)a\left(u\left( t\right) -u\left(\tau\right),u\left(t\right) -u\left(\tau\right)\right)d\tau,
	\end{equation*}
	then,
	\begin{eqnarray}
		&&I\left( t\right) =\left( 1-\int_{0}^{t}r\left( \tau\right) d\tau\right) a\left( u,u\right) -\int_{\Omega}\ln\left\vert u\right\vert  \left\vert u\right\vert^{p}dx,\notag\\
		&&J\left( t\right) =\frac{1}{2}\left( 1-\int_{0}^{t}r\left( \tau\right) d\tau\right) a\left( u,u\right) -\frac{1}{p}\int_{\Omega}\ln\left\vert u\right\vert  \left\vert u\right\vert^{p}dx,\notag\\
		&&E\left( t\right) =J\left( t\right)+\frac{1}{2}\left\Vert u_{t}\right\Vert_{\Omega}^{2}+\frac{\mu}{2}\left\Vert \nabla u_{t}\right\Vert_{\Omega}^{2}+\frac{1}{2}\left( r\circ u\right) \left( t\right)  +\frac{1}{p^{2}}\left\Vert u\right\Vert_{p}^{p},\label{b001}
	\end{eqnarray}
	and where the energy functional is denoted by $E\left( t\right).$	
	\begin{lemma}
		Assume that conditions $\left( H_{1}\right)$ and $\left( H_{2}\right)$ are satisfied, the energy functional satisfies along the solution of \eqref{1} the estimate
		\begin{equation}
			E^{\prime}\left( t\right) =-\frac{r\left( t\right)}{2}a\left( u,u\right)+\frac{1}{2}\left( r\circ u\right) \left( t\right)  \leq0,\quad\text{ for all } t\in\left[ 0,T\right].\label{c01}
		\end{equation}
	\end{lemma}
	\begin{proof}
		Muliplying $\left(P\right)_{1} $ by $u_{t}$ and integrating over $\Omega$, and then applying integration by parts and hypothesis $\left(H_{1}\right)$ and $\left(H_{2}\right),$ we derive equation $\left( \ref{c01}\right) $.
	\end{proof}
	\par To get a general stability result, the following is needed.
	\begin{remark}
		\label{remark 3.9} 1) As in \cite{11}, under the assumptions  $\left( H_{1}\right)$ and $\left( H_{2}\right)$, it follows that $\lim\limits_{t \rightarrow +\infty} r(t) = 0.$ Hence, there exists a sufficiently large $t_{0} > 0$ such that
		\begin{equation*}
			r\left( t_{0}\right)=r_{0}\implies r\left( t\right)  \leq r_{0}, \quad\text{ for all } t \in\left[  t_{0},\infty\right) .
		\end{equation*}
		Since $r$ and $h$ are positive, nonincreasing, continuous functions, and $\mathcal{G}$ is also a positive continuous function, for all $t \in \left[ 0, t_{0} \right]$, we obtain
		\begin{equation*}
			r^{\prime}\left( t\right) \leq -h\left( t\right) \mathcal{G}\left( r\left( t\right) \right) \leq -\frac{c_{2}}{r\left( 0\right) }r\left( 0\right) \leq -\frac{c_{2}}{r\left( 0\right) }r\left( t\right),
		\end{equation*}
		which leads to
		\begin{equation}
			r\left( t\right) \leq-\frac{r\left( 0\right) }{c_{2}}r^{\prime}\left( t\right),\label{c001}
		\end{equation}
		where $c_{2}$ is a positive constants.\\
		$2)$ If $\mathcal{G}$ is a strictly convex and $\mathcal{G}\left( 0 \right)= 0,$ then
		\begin{equation}
			\mathcal{G}\left( s\tau\right)\leq \tau \mathcal{G}\left( s\right),\quad \text{ for}\quad\tau\in\left[ 0,1\right]\quad  \text{ and }\quad s\in\left( 0,r_{0}\right].\label{0159}
		\end{equation}
		\par Moreover, if  $\mathcal{G}$ is strictly increasing $C^{2}$ function on $\left( 0,r_{0}\right]$ and satisfies \eqref{789}, then it has an extension $\overline{\mathcal{G}}$ in $C^{2}\left( \mathbb{R}_{+}^{*};\mathbb{R}_{+} \right)$, that is strictly convex and strictly increasing. For instance, we can define
		$\overline{\mathcal{G}},$ for any $t> r_{0},$ by  	
		\begin{equation*}
			\overline{\mathcal{G}}\left( t\right) =\frac{a_{1}}{2}t^{2}+\left( a_{0}-a_{1}r_{0}\right) t+\left(\frac{a_{1}}{2}r_{0}^{2}-a_{0}r_{0}+a_{2}\right),
		\end{equation*}
		where $\mathcal{G}^{\left( i\right) }\left( r_{0}\right) =a_{i},$ for all $ i\in\{0,1,2\}$.\\
		$3)$ Let $\mathcal{G}^{*}$ denote the convex conjugate of $\mathcal{G}$ according to Young's definition \cite{12}, then
		\begin{equation}
			\overline{\mathcal{G}}^{*}\left( \tau\right)= \tau\left(\overline{\mathcal{G}}^{\prime} \right) ^{-1}\left( \tau\right)- \overline{\mathcal{G}}\left[ \left(\overline{\mathcal{G}}^{\prime} \right) ^{-1}\left( \tau\right)\right] , \text{ if } \tau \in \left( 0,\overline{\mathcal{G}}^{\prime}\left( r_{0}\right)\right] ,\label{00159}
		\end{equation}
		and $\overline{\mathcal{G}}^{*}$ satisfies the following Young’s inequality
		\begin{equation}
			AB \leq \overline{\mathcal{G}}^{*}\left( A\right)+\overline{\mathcal{G}}\left( B\right) , \text{ if } A \in \left( 0,\overline{\mathcal{G}}^{\prime}\left( r_{0}\right) \right] , B\in \left( 0,r_{0}\right].\label{1590}
		\end{equation}
		$4)$ If $Q$ is a convex function on $\left[c, d \right],$ $f:\Omega\longrightarrow \left[c, d\right]$ and $\chi$ are integrable functions on $\Omega,$ $\chi\left(x\right)\geq 0,$ and $\int_\Omega \chi\left(x\right)dx=\chi_{0}>0,$ then Jensen's inequality asserts that
		\begin{equation*}
			Q\left[ \frac{1}{\chi_{0}}\int_\Omega f\left( x\right)\chi\left(x\right)dx\right] \leq   \frac{1}{\chi_{0}}\int_\Omega Q\left[f\left(x\right)\right]\chi\left(x\right)dx.\label{jensen}
		\end{equation*}
	\end{remark}
	\begin{lemma}{\cite{x7}}
		\label{lem22}
		For every positive $\mu_{0},$ we have 
		\begin{equation*}
			\left\vert s^{\mu_{0}}\ln s\right\vert\leq\frac{1}{e\mu_{0}}\quad \text{for}\quad s\in\left( 0,1\right)  \quad \text{and} \quad 0\leq \ln s\leq\frac{s^{\mu_{0}}}{e\mu_{0}}\quad\text{for}\quad s\in \left[  1,\infty\right) .
		\end{equation*}
	\end{lemma}
	\begin{lemma}
		\label{lem3.4}Assume that $\left( H_{1}\right) $ and $\left( H_{2}\right)$ hold and $u_{0}\in H^{2}\left( \Omega\right)$  and $u_{1}\in H^{1}\left( \Omega \right).$ If
		\begin{equation}
			I\left( 0\right)>0\quad\text{and}\quad \beta=\frac{B_{p+\mu_{0}}^{p+\mu_{0}}}{c_{0}\mu_{0}el}\left(\frac{2p}{c_{0}\left( p-2\right)l }E\left( 0\right) \right) ^{\frac{p+\mu_{0}-2}{2}}<\frac{1}{2},\label{c20}
		\end{equation}
		then $I\left( t\right) >0.$
	\end{lemma}
	\begin{proof}
		Since $I\left( 0\right) >0,$ there exists by continuity a $T_{j}<T$ such that
		\begin{equation*}
			I\left( u\left( t\right) \right)\geq 0,\quad \forall t\in\left[ 0,T_{j}\right] ;	 	
		\end{equation*}
		which yields
		\begin{eqnarray}
			J\left( t\right) &=&\frac{1}{2}\left( 1-\int_{0}^{t}r\left( \tau\right) d\tau\right) a\left(u, u\right) -\frac{1}{p}\int_{\Omega}\ln\left\vert u\right\vert  \left\vert u\right\vert^{p}dx\notag\\
			&\geq&\frac{p-2}{2p}\left( 1-\int_{0}^{t}r\left( \tau\right) d\tau\right) a\left(u, u\right)+\frac{1}{p}I\left( t\right)\notag\\
			&\geq&\frac{p-2}{2p}\left( 1-\int_{0}^{t}r\left( \tau\right) d\tau\right) a\left(u, u\right).\label{3.15}
		\end{eqnarray}
		Thus, by employing $\left( H_{1}\right) ,\left( H_{2}\right), \left( \ref{c01}\right)$ and $\left( \ref{3.15}\right)$, it is readily apparent that
		\begin{eqnarray}
			la\left(u, u\right) \leq\left( 1-\int_{0}^{t}r\left( \tau\right) d\tau\right) a\left(u, u\right)\leq\left( \frac{2p}{p-2}\right) J\left( t\right) 
			\leq \left( \frac{2p}{p-2}\right) E\left( 0\right).\label{231}
		\end{eqnarray}
		By utilizing  $\left( H_{1}\right), \left( H_{2}\right), \left( \ref{103}\right), \left( \ref{105}\right)$ and $\left( \ref{231}\right)$ to obtain, for all $t\in\left[ 0,T_{j}\right] ,$\\
		we choose $\mu_{0}>0$ with
		\begin{equation*}
			p+\mu_{0}\in\left( 2,+\infty\right)
		\end{equation*}
		and recalling Lemma \ref{lem22}, we obtain
		\begin{eqnarray}
			\int_{\Omega}\ln\left\vert u\right\vert  \left\vert u\right\vert^{p}dx&\leq&\frac{1}{\mu_{0}e}\left< \left\vert u\right\vert^{p},\left\vert u\right\vert^{\mu_{0}}\right>_{\left\lbrace \left\vert u\right\vert\geq 1\right\rbrace }\notag\\
			&\leq&\frac{1}{\mu_{0}e}\left\Vert u\right\Vert_{p+\mu_{0}}^{p+\mu_{0}}\leq \frac{1}{\mu_{0}e}B_{p+\mu_{0}}^{p+\mu_{0}}\left\Vert \nabla u\right\Vert_{\Omega}^{p+\mu_{0}-2}\left\Vert \nabla u\right\Vert_{\Omega}^{2}\notag\\
			&\leq&\frac{B_{p+\mu_{0}}^{p+\mu_{0}}}{e\mu_{0}l}\left\Vert \nabla u\right\Vert_{\Omega}^{p+\mu_{0}-2}l\left\Vert \nabla u\right\Vert_{\Omega}^{2}
			\leq l\beta a\left(u, u\right) \notag\\
			&<&\left( 1-\int_{0}^{t}r\left( \tau\right) d\tau\right) a\left( u,u\right).\label{45po}\notag
		\end{eqnarray}
		Therefore,
		\begin{equation*}
			I\left( t\right) =\left( 1-\int_{0}^{t}r\left( \tau\right) d\tau\right) a\left(u, u\right)-	\int_{\Omega}\ln\left\vert u\right\vert  \left\vert u\right\vert^{p}dx>0.
		\end{equation*}
		Note that
		\begin{equation*}
			\underset{t\rightarrow T_{j}}{\lim }\dfrac{B_{p+\mu_{0}}^{p+\mu_{0}}}{c_{0}\mu_{0}el}\left( \frac{2p}{c_{0}\left( p-2\right)l }E\left( 0\right) \right) ^{\frac{p+\mu_{0}-2}{2}}<\frac{1}{2}.
		\end{equation*}%
		We repeat this procedure, $T_{j}$ is extended to $T$.
	\end{proof}	
	\begin{lemma}
		\label{lem03}
		Suppose that $\left( H_{1}\right)$ hold, $u_{0}\in H^{2}\left( \Omega\right) $ and $u_{1}\in H^{1}\left( \Omega \right).$ We have for a positive constant $\eta_{1},$ the following inequality
		\begin{equation*}
			\left\Vert \nabla u\right\Vert_{\Omega}^{\eta_{1}} \leq D_{0}^{\frac{\eta_{1}-2}{2}}\left\Vert \nabla u\right\Vert_{\Omega}^{2},\quad \forall t>0,
		\end{equation*}
		where $D_{0}=\dfrac{2p}{c_{0}\left( p-2\right)l}  E\left(0\right).$
	\end{lemma}
	\begin{proof}
		using $\left(\ref{105}\right)$, $\left(\ref{231} \right)$ and  Lemma $\ref{lem3.4}$, yield 
		\begin{equation*}
			\left\Vert \nabla u\right\Vert_{\Omega}^{2}\leq    \frac{2p}{c_{0}\left( p-2\right)l }  E\left(0\right),
		\end{equation*}
		then
		\begin{equation*}
			\left\Vert \nabla u\right\Vert_{\Omega}^{\eta_{1}}=\left\Vert \nabla u\right\Vert_{\Omega}^{\eta_{1}-2}\left\Vert \nabla u\right\Vert_{\Omega}^{2}\leq \left( \frac{2p}{c_{0}\left( p-2\right)l }  E\left(0\right)\right)  ^{\frac{\eta_{1}-2}{2}}\left\Vert \nabla u\right\Vert_{\Omega}^{2}= D_{0}^{\frac{\eta_{1}-2}{2}}\left\Vert \nabla u\right\Vert_{\Omega}^{2}.
		\end{equation*}
	\end{proof}
	\begin{lemma}
		\label{lem4}Suppose that $\left( H_{1}\right)$ and $\left( H_{2}\right)$ hold. The derivative of the  functional
		\begin{equation*}
			\Psi_{0}\left(t\right):=\left<u,u_{t}\right>_{\Omega}+\mu\left<\nabla u,\nabla u_{t}\right>_{\Omega},
		\end{equation*}
		satisfies the estimate
		\begin{equation}
			\Psi_{0}^{\prime}\left(t\right)\leq-\frac{l}{2}a\left( u,u\right)+\left\Vert u_{t}\right\Vert _{\Omega}^{2}+\mu\left\Vert \nabla u_{t}\right\Vert _{\Omega}^{2}+\frac{C_{\alpha}}{2l}\left(g_{\alpha}\circ u\right) \left( t\right)+\int_{\Omega} \ln\left\vert u\right\vert\left\vert u \right\vert^{p}dx, \label{024}
		\end{equation}
		for any $0<\alpha<1$, where
		\begin{equation}
			C_{\alpha}:=\int_{0}^{\infty}\frac{r^{2}\left( \tau\right) }{\alpha r\left( \tau\right) -r^{\prime}\left( \tau\right) }d\tau \quad\text{ and }\quad g_{\alpha}\left( \tau\right):=\alpha r\left( \tau\right) -r^{\prime}\left( \tau\right).\label{ertz}
		\end{equation}
	\end{lemma}
	\begin{proof} 
		It's evident from applying the first equation in $ \left(P\right) $, $\left( \ref{107}\right)$ and $\left( \ref{106}\right)$, we obtain
		\begin{eqnarray}
			\Psi_{0}^{\prime}\left( t\right) 
			&\leq&-\left( 1-\int_{0}^{\infty}r\left(\tau\right)d\tau\right)a\left(u,u\right)+\left\Vert u_{t}\right\Vert _{2}^{2}+\mu\left\Vert\nabla u_{t}\right\Vert_{\Omega}^{2}+\int_{\Omega} \ln\left\vert u\right\vert\left\vert u \right\vert^{p}dx  \notag\\
			&&+\int_{0}^{t}r\left(t-\tau\right) a\left(u\left( \tau\right) -u\left( t\right),u\left( t\right) \right)d\tau\notag \\
			&\leq&\left\Vert u_{t}\right\Vert _{2}^{2}+\mu\left\Vert\nabla u_{t}\right\Vert_{\Omega}^{2}-\frac{l}{2}a\left(u,u\right)+\int_{\Omega} \ln\left\vert u\right\vert\left\vert u \right\vert^{p}dx\notag\\
			&&+\frac{1}{2l}\int_{0}^{t}r\left(t-\tau\right)\int_{0}^{t}r\left(t-\tau\right) a\left(u\left( \tau\right) -u\left( t\right),u\left( \tau\right) -u\left( t\right) \right)d\tau d\tau.\label{24}	\quad\quad
		\end{eqnarray}
		By using Cauchy-Schwarz inequality and \eqref{ertz}, we have
		\begin{eqnarray}
			&&\int_{0}^{t}r\left(t-\tau\right)\int_{0}^{t}r\left(t-\tau\right) a\left(u\left( \tau\right) -u\left( t\right),u\left( \tau\right) -u\left( t\right) \right)d\tau d\tau\notag\\
			&\leq&\int_{0}^{t}\left( \frac{r^{2}\left( \tau\right) }{\alpha r\left( \tau\right) -r^{\prime}\left( \tau\right) }d\tau\right)  \int_{0}^{t}g_{\alpha}\left(t-\tau\right) a\left(u\left( \tau\right) -u\left( t\right),u\left( \tau\right) -u\left( t\right) \right)d\tau\notag\\
			&\leq&C_{\alpha}\left( g_{\alpha}\circ u\right) \left( t\right) .\label{cdfb}
		\end{eqnarray}
		From \eqref{24} and \eqref{cdfb}, we get the estimate \eqref{024}.
	\end{proof}
	\begin{lemma}
		\label{lem5}
		Suppose  that $\left( H_{1}\right)$ and $\left( H_{2}\right)$ hold. The derivative of the functional 
		\begin{equation*}
			\Psi_{1} \left( t\right):=-\left<u_{t},\int_{0}^{t}
			r\left(t-\tau\right)\left( u\left( t\right) -u\left(\tau\right) \right) ds\right>_{\Omega}
			-\mu\left<\nabla u_{t},\int_{0}^{t}
			r\left(t-\tau\right)\nabla\left(u\left( t\right) - u\left(\tau\right) \right) d\tau\right>_{\Omega},
		\end{equation*}
		satisfies the estimate 
		\begin{eqnarray}
			\Psi_{1}^{\prime}\left( t\right)&\leq& \sigma c_{3} a\left( u,u\right)
			-\left(\int_{0}^{t}r\left(\tau\right) d\tau-\sigma\right)\left( 
			\left\Vert u_{t}\right\Vert_{\Omega}^{2}+\mu\left\Vert \nabla u_{t}\right\Vert_{\Omega}^{2}\right)\notag\\
			&&+\left(C_{\alpha}+\dfrac{c_{4}+c_{5}C_{\alpha}}{2\sigma c_{0}c_{1}} \right) \left( g_{\alpha}\circ u\right) \left( t\right) , \label{30}
		\end{eqnarray}
		where $
		c_{3}=\dfrac{c_{0}+C_{E\left( 0\right) }}{2c_{0}},c_{4}=\left( c_{0}+\mu c_{1}\right) \left( r\left( 0\right) +r\left( 1-l\right) \right),c_{5}=c_{0}\left( c_{1}+1\right) +\left( c_{0}+\mu c_{1}\right)\alpha^{2},$ and
		\begin{equation}
			C_{E(0)}= \left( \frac{B_{2\left(p-\mu_{3}-1\right)}^{  p-\mu_{3}-1}}{e\mu_{3}}\right) ^{2}D_{0}^{ p-\mu_{3} -2}+\left( \frac{B_{2 \left(p+\mu_{4}-1\right)  }^{p+\mu_{4}-1} }{e\mu_{4}}\right) ^{2}D_{0}^{ p+\mu_{4}-2},\label{z0}
		\end{equation}
		for any  $\sigma\in \left( 0,1\right)$ and $D_{0}$ was introduced in Lemma $\ref{lem03}$.
		
	\end{lemma}
	\begin{proof}
		With the help of equation $\left(P\right)_{1}$, integrating by parts $\left( \ref{107}\right) $, and some modifications, we arrived at 
		\begin{eqnarray}
			\Psi_{1}^{\prime}\left( t\right)&=&\left( 1-\int_{0}^{t}r\left( \tau\right) d\tau\right)  \int_{0}^{t}r\left(t-\tau\right)a\left(   u\left( t\right) -u\left( \tau\right) ,u\left( t\right) \right)  d\tau\notag\\
			&&+\int_{0}^{t}r\left( t-\tau\right)\int_{0}^{t}r\left( t-\tau\right) a\left(  u\left( t\right) -u\left( \tau\right),u\left( t\right) -u\left( \tau\right) \right)  d\tau d\tau\notag\\
			&&- \int_{0}^{t}r\left(t-\tau\right) \left<\left\vert u\right\vert^{p-2}u\ln\left\vert u\right\vert ,u\left( t\right) -u\left(\tau\right)\right>_{\Omega}d\tau \notag\\
			&&-\int_{0}^{t}r^{\prime}\left( t-\tau\right) \left<u_{t}\left( t\right) , u\left( t\right) -u\left( \tau\right) \right>_{\Omega}  d\tau-\left( \int_{0}^{t}r\left(\tau\right) d\tau\right) \left\Vert u_{t}\right\Vert _{\Omega}^{2}\notag\\
			&&-\mu \int_{0}^{t}r^{\prime}\left(t-\tau\right) \left<\nabla u_{t}\left( t\right) ,\nabla u\left( t\right) -\nabla u\left(\tau\right) \right>_{\Omega}d\tau-\mu\left( \int_{0}^{t}r\left(\tau\right) d\tau\right) \left\Vert \nabla u_{t}\right\Vert _{\Omega}^{2}. \notag\\
			&=&J_{1}\left( t\right)+J_{2}\left( t\right)+J_{3}\left( t\right)+J_{4}\left( t\right)+J_{5}\left( t\right)-\left( \int_{0}^{t}r\left(\tau\right) d\tau\right) \left( \left\Vert u_{t}\right\Vert _{\Omega}^{2}+\mu\left\Vert \nabla u_{t}\right\Vert _{\Omega}^{2} \right).\notag \label{030}
		\end{eqnarray}
		By using \eqref{104}, \eqref{106},  \eqref{842}, \eqref{b001} and \eqref{c01}, we get the next result
		\begin{eqnarray}
			&&\left\vert J_{1}\left( t\right) \right\vert\leq\frac{\sigma}{2}a\left( u,u\right)+  \frac{C_{\alpha}}{2\sigma}\left(g_{\alpha}\circ u\right) \left( t\right),\label{031} \\
			&&\left\vert J_{2}\left( t\right) \right\vert\leq c_{\alpha}\left(g_{\alpha}\circ u\right) \left( t\right)\\
			&& \left\vert J_{4}\left( t\right) \right\vert\leq\sigma\left\Vert u_{t}\right\Vert_{ \Omega }^{2}+\frac{1}{2\sigma c_{1}}\left( r\left( 0\right) +\alpha\left( 1-l\right) +\alpha^{2}C_{\alpha}\right) \left( g_{\alpha}\circ u\right) \left( t\right) \\
			&& \left\vert J_{5}\left( t\right) \right\vert\leq\sigma\mu\left\Vert \nabla u_{t}\right\Vert_{ \Omega }^{2}+\frac{\mu}{2\sigma c_{0}}\left( r\left( 0\right) +\alpha\left( 1-l\right) +\alpha^{2}C_{\alpha}\right) \left( g_{\alpha}\circ u\right)\left( t\right) .\label{dfg}
		\end{eqnarray}
		Based on Young's inequality, we have for all $\sigma\in\left( 0,1\right)$
		\begin{equation}
			\left\vert J_{3}\left( t\right)\right\vert \leq \frac{\sigma}{2} \left\Vert  \left\vert u\right\vert^{p-2}u\ln\left\vert u\right\vert\right\Vert _{\Omega} ^{2}+\frac{1}{2\sigma}  \left\Vert\int_{0}^{t}r\left( t-\tau\right)\left( u\left( t\right) -u\left( \tau\right) \right) d\tau\right\Vert_{\Omega}^{2}.
		\end{equation}
		Now, with the same notation of \cite{9}, let
		\begin{equation*}
			\Omega_{3}=\left\lbrace x\in \Omega :\left\vert u\right\vert\in \left[ 0,1\right)\right\rbrace \quad \text{and}\quad \Omega_{4}=\left\lbrace x\in \Omega :\left\vert u\right\vert\in\left[ 1,\infty\right) \right\rbrace .
		\end{equation*}
		Since to $2<2\left( p-2\right) <+\infty$, there exist $\mu_{3}>0$ and $\mu_{4}>0$ sucht that $2<2\left( p-\mu_{3}-1\right) <+\infty$ and $2<2\left( p+\mu_{4}-1\right) <+\infty,$ respectively. Using Lemma $\ref{lem22}$, $\left(\ref{103} \right) $, Lemma $\ref{lem03}$, $\left(\ref{105} \right) $, and $\left(\ref{z0}\right)$, we obtain
		\begin{eqnarray}
			\left\Vert \left\vert u\right\vert^{p-2}u\ln\left\vert u\right\vert\right\Vert_{\Omega}^{2}
			&\leq&\left( \frac{1}{e\mu_{3}}\right)^{2} \left\Vert u\right\Vert_{\Omega_{3}}^{2\left( p -\mu_{3}-1\right) }+\left( \frac{1}{e\mu_{4}}\right)^{2} \left\Vert u\right\Vert_{\Omega_{4}}^{2\left( p+\mu_{4}-1\right) }\notag\\
			&\leq& \left( \frac{B_{2\left(p-\mu_{3}-1\right) }^{  p-\mu_{3}-1}}{e\mu_{3}}\right)^{2}\left\Vert \nabla u\right\Vert_{\Omega}^{2 \left(p-\mu_{3}-1\right) } +\left( \frac{B_{2\left( p+\mu_{4}-1\right)  }^{p+\mu_{4}-1} }{e\mu_{4}}\right)^{2}\left\Vert \nabla u\right\Vert_{\Omega}^{2 \left( p+\mu_{4}-1\right) }\notag\\
			&\leq& \left[ \left( \frac{B_{2\left(p-\mu_{3}-1\right)}^{  p-\mu_{3}-1}}{e\mu_{3}}\right) ^{2}D_{0}^{ p-\mu_{3} -2}+\left( \frac{B_{2 \left(p+\mu_{4}-1\right)  }^{p+\mu_{4}-1} }{e\mu_{4}}\right) ^{2}D_{0}^{ p+\mu_{4}-2}\right]\left\Vert \nabla u\right\Vert_{\Omega}^{2}\notag\\
			&\leq&\frac{C_{E\left( 0\right)}}{c_{0}} a\left(u,u\right) .\label{033}
		\end{eqnarray}
		By the Young’s inequality, Cauchy-Schwarz inequality,  \eqref{104}, \eqref{ertz} and \eqref{033}, we obtain
		\begin{equation}
			\left\vert J_{3}\left( t\right) \right\vert\leq\frac{\sigma C_{E\left( 0\right)} }{2c_{0}}a\left( u,u\right)+\frac{C_{\alpha}}{2\sigma c_{1}}\left(g_{\alpha}\circ u\right) \left( t\right).\label{kjlm}
		\end{equation}
		Combining $\left(\ref{031}\right)-\left(\ref{dfg}\right)$ and $\left(\ref{kjlm}\right)$, Lemma $\ref{lem5}$ is proved.
	\end{proof}
	
	\begin{lemma}
		Under the assumption $\left( H_{1}\right) $, the functional $\beta\left( t\right) $ defined by
		\begin{equation*}
			\beta\left( t\right) =\int_{0}^{t}k\left( t-\tau\right) a\left(u\left( \tau\right),u\left( \tau\right)  \right) d\tau,
		\end{equation*}
		satisfies the estimate
		\begin{equation}
			\beta^{\prime}\left( t\right) \leq 4\left( 1-l\right) a\left( u,u\right) -\frac{3}{4}\left( r\circ u\right) \left( t\right) ,\label{cv1}
		\end{equation}
		where, $k\left( t\right) =\int_{t}^{+\infty}r\left( \tau\right) d\tau.$ 
	\end{lemma}
	\begin{proof}
		Using $\left( \ref{106}\right) ,\left( \ref{842}\right)$
		and the fact that $k\left( t\right) \leq k\left( 0\right)=1-l$, we see that
		\begin{eqnarray}
			\beta^{\prime}\left( t\right) &=&k\left( 0\right) a\left( u,u\right) -\int_{0}^{t}r\left( t-\tau\right) a\left( u\left( \tau\right) ,u\left( \tau\right) \right) d\tau\notag\\
			&\leq&\left(k\left( 0\right) -\int_{0}^{t}r\left( \tau\right) d\tau\right) a\left( u,u\right) -\left( r\circ  u\right) \left( t\right) -2\int_{0}^{t}r\left( t-\tau\right) a\left( u\left( \tau\right) -u\left( t\right) ,u\left( t\right) \right) d\tau\notag\\
			&\leq&4\left( 1-l\right) a\left( u,u\right) -\frac{3}{4}\left( r\circ u\right) \left( t\right).\notag
		\end{eqnarray} 
	\end{proof}
	\begin{lemma}
		Suppose that $\left(H_{1}\right), \left( H_{2}\right)$ and $\left(\ref{c20}\right)$ hold. Then, there exist constants $N, N_{0}, N_{1}>0$,such that the functional
		\begin{equation*}
			\mathcal{L}\left( t\right) :=NE\left( t\right) +N_{0}\Psi_{1}\left( t\right) +N_{1}\Psi_{0}\left( t\right) ,
		\end{equation*}
		satisfies the asymptotic equivalence
		\begin{equation*}
			E\sim \mathcal{L}.
		\end{equation*}
		Moreover, for a suitable choice of the positive constants $N,N_{0}$ and $N_{1},$ the functional $\mathcal{L}\left( t\right)$ satisfies the following estimates
		\begin{equation}
			\mathcal{L}^{\prime}\left( t\right)\leq -\left\Vert u_{t}\right\Vert_{\Omega}^{2}-\mu\left\Vert \nabla u_{t}\right\Vert_{\Omega}^{2}-5\left( 1-l\right) a\left( u,u\right) +\frac{1}{2}\left( r\circ u\right) \left( t\right) , \quad\forall t\in\left[  t_{1},\infty\right) .\label{034}
		\end{equation}
	\end{lemma}
	\begin{proof}
		Given that $r$ is positive and $r\left( 0\right) >0,$ then for any  $t_{0}>0,$ we have the following:
		\begin{equation*}
			\int_{0}^{t}r\left(\tau\right) d\tau\geq \int_{0}^{t_{0}}r\left( \tau\right) d\tau=r_{0},\quad t\geq t_{0}.
		\end{equation*}
		By utilizing $\left( \ref{c01}\right),\left( \ref{c20}\right),\left( \ref{024}\right),\left( \ref{30}\right) $ and setting $\sigma=\frac{l}{2c_{2}N_{1}},$ it becomes evident that
		\begin{eqnarray}
			\mathcal{L}^{\prime}\left( t\right)& \leq&-\left(N_{1}r_{0}-N_{0}-\frac{l}{2c_{2}}\right) \left\Vert u_{t}\right\Vert _{2}^{2}-\mu\left(N_{1}r_{0}-\frac{l}{2c_{2}}\right) \left\Vert \nabla u_{t}\right\Vert_{\Omega}^{2}\notag\\
			&& -\left(N_{0}\left( \frac{l}{2}-l\beta \right) -\frac{l}{2}\right)a\left(u,u\right)+\frac{N\alpha}{2} \left(r\circ u \right)\left(t \right)\notag\\
			&&-\left( \frac{N}{2}-\frac{c_{3}}{lc_{1}}N_{1}^{2}-C_{\alpha}\left( \frac{N_{0}}{2l}+\frac{c_{4}}{lc_{1}}N_{1}^{2}+N_{1}\right)\right)  \left( g_{\alpha}\circ u\right) \left( t\right).\notag
		\end{eqnarray}
		Presently, we select $N_{0}$ to be sufficiently large such that 
		\begin{equation}
			N_{0}\left( \frac{l}{2}-l\beta\right) -\frac{l}{2}>5\left( 1-l\right) ,\label{poi}
		\end{equation}
		then we choose $N_{1}$ to be sufficiently large such that 
		\begin{equation}
			N_{1}r_{0}-N_{0}-\frac{l}{2c_{2}}>1.\label{poj}
		\end{equation}
		From \eqref{789} and \eqref{ertz}, we get
		\begin{equation*}
			\frac{\alpha r^{2}\left( t\right) }{\alpha r\left( t\right)-r^{\prime}\left( t\right)}\leq r\left( t\right)\implies \alpha C_{\alpha}=\alpha\int_{0}^{\infty}\frac{r^{2}\left( t\right) }{\alpha r\left( t\right)-r^{\prime}\left( t\right)}d\tau<\left( 1-l\right) .\label{rezt}
		\end{equation*}
		By the Lebesgue dominated convergence theorem, we find that $\alpha C_{\alpha}\rightarrow 0 $ as $\alpha \rightarrow 0.$ Then, there
		exists $0<\alpha_{0}<1$ such that, if $\alpha<\alpha_{0},$ then
		\begin{equation*}
			C_{\alpha}\left( \frac{N_{0}}{2l}+\frac{c_{4}}{lc_{1}}N_{1}^{2}+N_{1}\right)<\frac{1}{4}.
		\end{equation*}
		Finally, by selecting $\alpha=\frac{1}{2N}<\alpha_{0}$ and take $N$ large enough so that
		\begin{equation}
			\frac{N}{2}-\frac{c_{3}}{lc_{1}}N_{1}^{2}-C_{\alpha}\left( \frac{N_{0}}{2l}+\frac{c_{4}}{lc_{1}}N_{1}^{2}+N_{1}\right)>0.\label{pok}
		\end{equation}
		Combining \eqref{poi}, \eqref{poj} and \eqref{pok} we obtain \eqref{034}.
	\end{proof}
	\begin{lemma}
		Assume that $\left( H1\right) $ and $\left( H2\right) $ hold; then, we have,
		\begin{equation}
			\int_{0}^{+\infty}E\left( \tau\right) d\tau\leq +\infty.\label{mbv}
		\end{equation}
	\end{lemma}
	\begin{proof}
		We consider a nonnegative function 
		\begin{equation*}
			G\left( t\right) =\mathcal{L}\left( t\right) +\beta\left( t\right).
		\end{equation*}
		From $\left( \ref{b001}\right) ,\left( \ref{cv1}\right) $ and $\left( \ref{034}\right) ,$ we obtain
		\begin{equation}
			G^{\prime}\left( t\right)\leq-\left\Vert u_{t}\right\Vert-\mu\left\Vert \nabla u_{t}\right\Vert-\left( 1-l\right) a\left( u,u\right) -\frac{1}{4}\left( r\circ u\right) \left( t\right)  \leq -\beta_{1}E\left( t\right),\label{mbc0}
		\end{equation}
		where $\beta_{1}$ is a positive constant. Integrating $\left( \ref{mbc0}\right) $ from $t_{0}$ to $t$, we obtain
		\begin{equation*}
			\beta_{1}\int_{t_{0}}^{t}E\left( \tau\right) d\tau\leq G\left( t_{0}\right) -G\left( t\right) \leq G\left( t_{0}\right) .
		\end{equation*}
		Therefore, we conclude that
		\begin{equation*}
			\int_{0}^{+\infty}E\left( \tau\right) d\tau\leq +\infty.
		\end{equation*}
	\end{proof}	
	\par Let us define
	\begin{equation}
		\theta\left( t\right) :=-\int_{t_{0}}^{t}r^{\prime}\left( \tau\right) a\left( u\left( t\right) -u\left( t-\tau\right),u\left( t\right) -u\left( t-\tau\right)\right) d\tau\leq -cE^{\prime}\left( t\right) . \label{58onb}
	\end{equation}
	\begin{lemma}
		Under the assumptions and $\left( H1\right) -\left( H2\right) $, we have the following
		estimates for all $t\in\left[ t_{0},\infty\right),$  
		\begin{equation}
			\int_{t_{0}}^{t}r\left( \tau\right) a\left( u\left( t\right) -u\left( t-\tau\right),u\left( t\right) -u\left( t-\tau\right)\right) d\tau\leq \frac{1}{\gamma}\mathcal{G}\left( \frac{\gamma\theta\left( t\right) }{h\left( t\right) }\right) .\label{vbhy}
		\end{equation}
		where $\gamma\in\left( 0,1\right) $ and $\overline{\mathcal{G}}$ is an extension of $\mathcal{G}$ such that $\overline{\mathcal{G}}$ is strictly increasing
		and strictly convex $C^{2}$ function on $\left( 0,\infty\right);$ see Remark \ref{remark 3.9}.
	\end{lemma}
	\begin{proof}
		We define the functional $\eta\left( t\right) $ as
		\begin{equation*}
			\eta\left( t\right):=\gamma\int_{t_{0}}^{t} a\left( u\left( t\right) -u\left( t-\tau\right),u\left( t\right) -u\left( t-\tau\right)\right) d\tau,\quad\forall t\in \left[ t_{0},\infty\right).
		\end{equation*}
		Using the equations \eqref{b001} and \eqref{231}, we deduce that
		\begin{equation*}
			a\left( u,u\right)\leq\frac{2p}{l\left(p-2\right)}E\left( t\right).
		\end{equation*}
		This estimate and \eqref{mbv} yield
		\begin{equation*}
			\int_{t_{0}}^{t}a\left( u\left( t\right) -u\left( t-\tau\right),u\left( t\right) -u\left( t-\tau\right)\right) d\tau\leq\frac{8p}{l\left( p-2\right) }\int_{t_{0}}^{\infty}E\left( \tau\right)d\tau<+\infty,\quad\forall t\geq t_{0}. 
		\end{equation*}
		So, with $\gamma\in\left( 0,1\right),$ we obtain  $\eta\left( t\right) \in\left( 0,\infty\right),$ for all $t\in\left[ t_{0},\infty\right)$. \\
		In view of the assumptions $\left( H1\right) ,\left( H2\right)$ and Jensen’s inequality, we obtain
		\begin{eqnarray}
			\theta\left( t\right) &=&-\frac{1}{\eta\left(t\right)}\int_{t_{0}}^{t}\eta\left( t\right)  r^{\prime}\left( \tau\right) a\left( u\left( t\right) -u\left( t-\tau\right) ,u\left( t\right) -u\left( t-\tau\right) \right)d\tau\notag \\
			&\geq& \frac{1}{\eta \left( t\right) }\int_{t_{0}}^{t}\eta\left( t\right)  h\left( \tau\right)\mathcal{G}\left(r\left(\tau\right)\right) a\left( u\left( t\right) -u\left( t-\tau\right) ,u\left( t\right) -u\left( t-\tau\right) \right) d\tau\notag \\ 
			&\geq& \frac{h\left( t\right) }{\gamma }\overline{\mathcal{G}}\left(\gamma\int_{t_{0}}^{t} r\left(\tau\right)a\left( u\left( t\right) -u\left( t-\tau\right) ,u\left( t\right) -u\left( t-\tau\right) \right)d\tau\right).\notag 
		\end{eqnarray}
		This yields, for any $t\geq t_{0}$
		\begin{equation*}
			\int_{t_{0}}^{t}r\left(\tau\right)a\left( u\left( t\right) -u\left( t-\tau\right) ,u\left( t\right) -u\left( t-\tau\right) \right)d\tau \leq\frac{1}{\gamma}\overline{\mathcal{G}}^{-1}  \left(\frac{\gamma \theta\left( t\right) }{h\left( t\right) }\right).
		\end{equation*}
	\end{proof}
	\par Now, we are ready to prove our main result.
	\begin{theorem}
		Assuming conditions $\left( H_{1}\right) $ and $\left( H_{2}\right) $ hold, there exist positive constants $k_{0}$ and $k_{1}$ such that \\
		\textbf{If $\mathcal{G}$ is linear:}
		\begin{equation*}
			E\left( t\right) \leq k_{1}e^{-k_{0}\int_{t_{0}}^{t}h\left( \tau\right) d\tau},\quad\forall t\geq t_{0}.
		\end{equation*}
		\textbf{If $\mathcal{G}$ is nonlinear:}
		\begin{equation*}
			E\left( t\right) \leq k_{1}\mathcal{G}_{1}^{-1}\left( \frac{k_{0}}{\int_{t_{0}}^{t}h\left( \tau\right) d\tau}\right),\label{theoremprincipale}\quad\quad\forall t\geq t_{0},
		\end{equation*}
		where 
		\begin{equation*}
			\mathcal{G}_{1}\left(\tau\right) =\tau\mathcal{G}^{\prime}\left(\tau\right) .
		\end{equation*}
	\end{theorem}
	\begin{proof}
		Since $h$ is nonincreasing, with the identity \eqref{c01} and \eqref{c001} we have 
		\begin{eqnarray}
			&&\int_{0}^{t_{0}}r\left( \tau\right) a\left( u\left( t\right) -u\left( t-\tau\right), u\left( t\right) -u\left( t-\tau\right)
			\right) d\tau\notag\\
			&&\leq -\frac{r\left( 0\right) }{c_{2}h\left( t_{0}\right) }\int_{0}^{t_{0}} r^{\prime}\left( \tau\right) a\left( u\left( t\right) -u\left( t-\tau\right), u\left( t\right) -u\left( t-\tau\right)\right) d\tau\leq -cE^{\prime}\left( t\right),\forall t\in\left[ t_{0},\infty\right).\hspace{1cm}\label{rgn}
		\end{eqnarray}
		Combining \eqref{vbhy}, \eqref{rgn} and  \eqref{034}, we obtain
		\begin{equation}
			H^{\prime}\left( t\right) \leq -mE\left( t\right) +\frac{c}{\gamma}\overline{\mathcal{G}}^{-1}\left( \frac{\gamma\theta\left( t\right) }{h\left( t\right) }\right) ,\quad\forall t\in\left[ t_{0},\infty\right),\label{58po}
		\end{equation}
		where $H:=\mathcal{L}+cE.$		
		\par \textbf{Case of $\mathcal{G}$ is linear:}\\
		Multiplying \eqref{58po} by $h\left( t\right),$ we obtain the following inequality
		\begin{equation*}
			h\left( t\right) H^{\prime}\left( t\right) \leq -mh\left( t\right) E\left( t\right) +c\theta\left( t\right),\quad\forall t\geq t_{0}.
		\end{equation*}
		Next, by applying equations \eqref{c01} and \eqref{58onb}, we deduce
		\begin{equation*}
			h\left( t\right) H^{\prime}\left( t\right) \leq -mh\left( t\right) E\left( t\right) -cE^{\prime}\left( t\right),\quad\forall t\geq t_{0}.
		\end{equation*}
		Since $h\left( t\right) $ is nonincreasing, this leads to
		\begin{equation*}
			\left( hH+cE\right) ^{\prime}\left( t\right) \leq -mh\left( t\right) E\left( t\right),\quad\forall t\geq t_{0}.
		\end{equation*}
		Utilizing the equivalence relation $\left( hH+cE\right)\sim E,$ we then directly obtain the following estimate
		\begin{equation*}
			E\left( t\right) \leq k_{2}e^{-k_{1}\int_{t_{0}}^{t}h\left( \tau\right) d\tau}.
		\end{equation*}
		\par \textbf{Case of $\mathcal{G}$ is nonlinear:}\\
		Let $s_{1}\in\left( 0,r_{0}\right)$ and define the functional $H_{1}\left( t\right) $ by 
		\begin{equation*}
			H_{1}\left( t\right):=\overline{\mathcal{G}}^{\prime}\left( \frac{s_{1}E\left( t\right) }{E\left( 0\right) }\right) H\left( t\right) ,\quad \text{ for all } t\in\left[  t_{0},\infty\right).
		\end{equation*}	
		Using the facts $E^{\prime}\leq 0,\mathcal{G}^{\prime}>0$ and $\mathcal{G}^{\prime\prime}>0,$ along with estimates from \eqref{58po}, we obtain 
		\begin{eqnarray}
			H_{1}^{\prime}\left( t\right) \leq \frac{s_{1}E\left( t\right) }{E\left( 0\right) }\overline{\mathcal{G}}^{\prime\prime}\left( \frac{s_{1}E\left( t\right) }{E\left( 0\right)}\right)H\left( t\right)+\overline{\mathcal{G}}\left( \frac{s_{1}E\left( t\right) }{E\left( 0\right)}\right)H^{\prime}\left( t\right),
		\end{eqnarray}	
		which leads to 
		\begin{equation*}
			H_{1}^{\prime}\left( t\right)\leq-mE\left( t\right) \overline{\mathcal{G}}^{\prime}\left( \frac{s_{1}E\left( t\right) }{E\left( 0\right)}\right)+\frac{c}{\gamma}\overline{\mathcal{G}}^{\prime}\left( \frac{s_{1}E\left( t\right) }{E\left( 0\right)}\right)\overline{\mathcal{G}}\left( \frac{\gamma \theta\left( t\right) }{h\left( t\right)}\right),\quad\forall t\geq t_{0}.
		\end{equation*}
		Applying equations \eqref{00159} and \eqref{1590}, with 
		\begin{equation*}
			A=\overline{\mathcal{G}}^{\prime}\left( \frac{s_{1}E\left( t\right) }{E\left( 0\right)}\right)\quad \text{ and } \quad B=\overline{\mathcal{G}}\left( \frac{\gamma \theta\left( t\right) }{h\left( t\right)}\right),
		\end{equation*}	
		yields 
		\begin{equation}
			H_{1}^{\prime}\left( t\right) \leq-mE\left( t\right) \overline{\mathcal{G}}^{\prime}\left( \frac{s_{1}E\left( t\right) }{E\left( 0\right)}\right)+\frac{c}{\gamma}\overline{\mathcal{G}}^{*}\left[ \overline{\mathcal{G}}^{\prime}\left( \frac{s_{1}E\left( t\right) }{E\left( 0\right)}\right)\right] + \frac{c \theta\left( t\right) }{h\left( t\right)}.
		\end{equation}
		Further simplification gives
		\begin{equation*}
			H_{1}^{\prime}\left( t\right) \leq-m\left( E\left( 0\right) -cs_{1}\right)\frac{E\left( t\right)}{E\left( 0\right) } \overline{\mathcal{G}}^{\prime}\left( \frac{s_{1}E\left( t\right) }{E\left( 0\right)}\right)+\frac{c \theta\left( t\right) }{h\left( t\right)},\quad\forall t\geq t_{0}.
		\end{equation*}
		Fixing $s_{1}$, we arrive at
		\begin{equation}
			H_{1}^{\prime}\left( t\right)\leq -m_{1}\frac{E\left( t\right)}{E\left( 0\right) } \overline{\mathcal{G}}^{\prime}\left( \frac{s_{1}E\left( t\right) }{E\left( 0\right)}\right)+\frac{c \theta\left( t\right) }{h\left( t\right)},\quad\forall t\geq t_{0}, \label{84ml}
		\end{equation}	
		where $m_{1}>0.$\\
		Multiplying both sides of this inequality by $h\left( t\right), $ and using the fact that $s_{1}\frac{E\left( t\right) }{E\left( 0\right) }<r_{0}$ along with equation \eqref{58onb}, leads to
		\begin{equation*}
			H_{1}^{\prime}\left( t\right)h\left( t\right) \leq -m_{1}\frac{E\left( t\right)}{E\left( 0\right) }\mathcal{G}^{\prime}\left( \frac{s_{1}E\left( t\right) }{E\left( 0\right)}\right)h\left( t\right)-c E^{\prime}\left( t\right) ,\quad\forall t\geq t_{0}.
		\end{equation*}	
		Let $H_{2}=hH_{1}+cE.$ Then, using the nonincreasing property of $h$, we derive the inequality	
		\begin{equation}
			m_{1}\frac{E\left( t\right)}{E\left( 0\right) }\mathcal{G}^{\prime}\left( \frac{s_{1}E\left( t\right) }{E\left( 0\right)}\right)h\left( t\right)\leq-H_{2}^{\prime}\left( t\right),\quad\forall t\geq  t_{0}.\label{nkie}
		\end{equation}
		The map 
		\begin{equation*}
			t\longmapsto E\left( t\right)  \mathcal{G}^{\prime}\left(  \frac{s_{1}E\left( t\right) }{E\left( 0\right)}\right) 
		\end{equation*}	
		is nonincreasing, since $\mathcal{G}^{\prime\prime}>0$ and $E\left(t\right)$ nonincreasing. As a result integrating the
		inequality \eqref{nkie} over $\left( t_{0},t\right) $ yields
		\begin{equation*}
			m_{1}\frac{E\left( t\right)}{E\left( 0\right) }\mathcal{G}^{\prime}\left( \frac{s_{1}E\left( t\right) }{E\left( 0\right)}\right)\int_{t_{0}}^{t}h\left( \tau\right)d\tau
			\leq H_{2}\left( t_{0}\right)-H_{2}\left( t\right)
			\leq H_{2}\left( t_{0}\right),\quad\forall t\geq  t_{0}.
		\end{equation*}	
		Finally, setting $\mathcal{G}_{0}\left( \tau\right)=\tau\mathcal{G}^{\prime}\left( \tau\right),$ we obtain the following estimate for some positive constants $k_{0}$ and $k_{1}$
		\begin{equation*}
			E\left( t\right) \leq k_{1}\mathcal{G}_{1}^{-1}\left( \frac{k_{0}}{\int_{t_{0}}^{t}h\left( \tau\right) d\tau}\right),\quad\forall t\geq t_{0}.
		\end{equation*}	
	\end{proof}


\begin{thebibliography}{99}
		
	
\bibitem{8} M.M. Al-Gharabli, A.M. Al-Mahdi, S.A. Messaoudi; General and optimal decay result for a viscoelastic problem with nonlinear boundary feedback, \textit{J. Dyn. Control Syst.}, \textbf{25} (2019), 551–572.

\bibitem{007} M.M. Al-Gharabli, A.M. Al-Mahdi, M. Noor, et al.; Numerical and theoretical stability study of a viscoelastic plate equation with nonlinear frictional damping term and a logarithmic source term, \textit{Math. Comput. Appl.}, \textbf{10} (2022), 1–27.

\bibitem{x0} F. Alabau-Boussouira; On convexity and weighted integral inequalities for energy decay rates of nonlinear dissipative hyperbolic systems, \textit{Appl. Math. Optim.}, \textbf{51} (2005), 61–105.

\bibitem{a} Y. Han, Q. Li; Lifespan of solutions to a damped plate equation with logarithmic nonlinearity, \textit{Evol. Equ. Control Theory}, \textbf{1} (2022), 25–40.

\bibitem{i} M. Santos, F. Junior; A boundary condition with memory for Kirchhoff plates equations, \textit{Appl. Math. Comput.}, \textbf{148} (2004), 475–496.

\bibitem{9} T.G. Ha, S.H. Park; Existence and general decay for a viscoelastic equation with logarithmic nonlinearity, \textit{J. Korean Math. Soc.}, \textbf{58} (2011), 1448–1499.

\bibitem{10} S.A. Messaoudi, W. Al-Khulaifi; General and optimal decay for a quasilinear viscoelastic equation, \textit{Appl. Math. Lett.}, \textbf{66} (2017), 1–22.

\bibitem{6} I. Baaziz, B. Benabderrahmane, S. Drabla; General decay results for a viscoelastic Euler–Bernoulli equation with logarithmic nonlinearity source and a nonlinear boundary feedback, \textit{Mediterr. J. Math.}, \textbf{20} (2023), 1–20.

\bibitem{0258} B. Madjour, A. Boudiaf; General stability result for a nonlinear viscoelastic wave equation with boundary dissipation, \textit{Nonlinear Dyn. Syst. Theory}, \textbf{25} (2025), 78–90.

\bibitem{02582} B. Madjour, A. Boudiaf, A. Merouani; On the well-posedness of the von Karman plate system with viscoelastic boundary damping, \textit{Gulf J. Math.}, \textbf{20} (2025), 176–189.

\bibitem{11} M.I. Mustafa; Optimal decay rates for the viscoelastic wave equation, \textit{Math. Methods Appl. Sci.}, \textbf{41} (2018), 192–204.

\bibitem{12} V.I. Arnold; \textit{Mathematical Methods of Classical Mechanics}, Springer-Verlag, New York, 1989.

\bibitem{c} K. Bartkowski, P. Górka; One-dimensional Klein–Gordon equation with logarithmic nonlinearities, \textit{J. Phys. A: Math. Theor.}, \textbf{41} (2008), 355201.

\bibitem{x6} M.I. Mustafa; Viscoelastic plate equation with boundary feedback, \textit{Evol. Equ. Control Theory}, \textbf{6} (2017), 261–276.

\bibitem{d} I. Bialynicki-Birula, J. Mycielski; Wave equations with logarithmic nonlinearities, \textit{Bull. Acad. Pol. Sci. Ser. Sci. Math. Astron. Phys.}, \textbf{23} (1975), 461–466.

\bibitem{f} J. Barrow, P. Parsons; In stationary models with logarithmic potentials, \textit{Phys. Rev. D}, \textbf{52} (1995), 5576–5587.

\bibitem{g} K. Enqvist, J. McDonald; Q-balls and baryogenesis in the MSSM, \textit{Phys. Lett. B}, \textbf{425} (1998), 309–321.

\bibitem{e} P. Górka; Logarithmic Klein–Gordon equation, \textit{Acta Phys. Pol. B}, \textbf{40} (2009), 59–66.

\bibitem{bm4} J.-Y. Park, J.-R. Kang; A boundary condition with memory for the Kirchhoff plate equations with nonlinear dissipation, \textit{Math. Methods Appl. Sci.}, \textbf{29} (2006), 267–280.

\bibitem{x1} M.M. Cavalcanti, V.N. Domingos Cavalcanti, T.F. Ma; Exponential decay of the viscoelastic Euler–Bernoulli equation with a nonlocal dissipation in general domains, \textit{J. Differ. Integral Equ.}, \textbf{17} (2004), 495–510.

\bibitem{x7} T.G. Ha, S.H. Park; Blow-up phenomena for a viscoelastic wave equation with strong damping and logarithmic nonlinearity, \textit{Adv. Differ. Equ.}, \textbf{2020} (2020), 235.


		
	\end{thebibliography}
\end{document}